\documentclass[12pt,a4paper]{amsart}
\usepackage{preamble}
\tikzset{edgee/.style = {->,> = latex'}}
\newcolumntype{P}[1]{>{\centering\arraybackslash}p{#1}}
\newcolumntype{M}[1]{>{\centering\arraybackslash}m{#1}}

\newcommand{\dsim}{\overset{d}{\sim}}
\newcommand{\hsim}{\overset{h}{\sim}}

\begin{document}

\title{Pattern avoidance and dominating compositions}

\author{Krishna Menon}
\address{Department of Mathematics, Chennai Mathematical Institute, India 603103}
\email{krishnamenon@cmi.ac.in}
\author{Anurag Singh}
\address{Department of Mathematics, Chennai Mathematical Institute, India 603103}
\email{anuragsingh@cmi.ac.in}

\begin{abstract}
Jel\'{\i}nek, Mansour, and Shattuck 
studied Wilf-equivalence among pairs of patterns of the form $\{\sigma,\tau\}$ where $\sigma$ is a set partition of size $3$ with at least two blocks.
They obtained an upper bound for the number of Wilf-equivalence classes for such pairs.
We show that their upper bound is the exact number of equivalence classes, thus solving a problem posed by them.
\end{abstract}

\maketitle

\section{Introduction}

For any $n \geq 1$, a partition of $[n]=\{1,\ldots,n\}$ is a collection of disjoint nonempty subsets of $[n]$ whose union is $[n]$.
The most common method of representing a set partition given by $\{B_1,B_2,\ldots,B_k\}$ is to write the blocks as
\begin{equation*}
    B_1/B_2/\cdots/B_k.
\end{equation*}
The usual convention followed is to order the blocks such that
\begin{equation*}
    \operatorname{min}(B_1) < \operatorname{min}(B_2) < \cdots < \operatorname{min}(B_k).
\end{equation*}Sometimes it is also convenient, when it does not cause confusion, to write the elements of each block without braces or commas.
For example, the partition $\{\{1,3,4\},\{2,6\},\{5,7,8\}\}$ of $[8]$ is written as $\{1,3,4\}/\{2,6\}/\{5,7,8\}$ or $134/26/578$.

A partition $B_1/B_2/\cdots/B_k$ of $[n]$ such that $\operatorname{min}(B_1) < \operatorname{min}(B_2) < \cdots < \operatorname{min}(B_k)$ can be represented by the sequence $a_1a_2\cdots a_n$ where $a_i=j$ if $i \in B_j$.
This sequence is called the \textit{restricted growth function} associated to the partition. For example, the partition $134/26/578$ has corresponding sequence $12113233$. For more on restricted growth functions, the interested reader is referred to Sagan's paper \cite[Section 4]{sag}. Henceforth, we will represent partitions using restricted growth functions.
The following definition of pattern avoidance in set partitions was first introduced by Sagan \cite{sag}.

\begin{definition}
A partition $\sigma=\sigma_1\sigma_2\cdots \sigma_n$ \textit{contains} a partition (or pattern) $\pi=\pi_1\pi_2\cdots\pi_m$ if there exists a subsequence $1 \leq h(1) < h(2) < \cdots < h(m) \leq n$ such that for any $i,j \in [m]$, $\sigma_{h(i)}=\sigma_{h(j)}$ if and only if $\pi_i=\pi_j$ and $\sigma_{h(i)}<\sigma_{h(j)}$ if and only if $\pi_i<\pi_j$. If $\sigma$ does not contain $\pi$, we say that $\sigma$ \textit{avoids} the pattern $\pi$.
\end{definition}

The topic of pattern avoidance has been an active area of research in enumerative combinatorics, starting with Knuth's work on permutations \cite{knu}. The study of pattern avoidance in set partitions was initiated by Klazar \cite{kla2}. Since then, several different notions of pattern avoidance of set partitions have been studied (see, e.g., the work of Chen et al. \cite{cddsy}, Goyt \cite{goyt}, or Bloom and Elizalde \cite{bloeli}). 

Let $T$ be a set of patterns. We will denote the set of partitions of $[n]$ that avoid all the patterns of $T$ as $P_n(T)$ and the number of such partitions, i.e., $|P_n(T)|$ as $p_n(T)$.

\begin{definition}
Two sets of patterns $T$ and $R$ are said to be Wilf-equivalent, written as $T \sim R$, if $p_n(T)=p_n(R)$ for all $n \geq 1$.
\end{definition}

For example, $P_n(\{123\})$ is the set of partitions of $[n]$ with at most two blocks and $P_n(\{122\})$ is the set of partitions where any block not containing $1$ is a singleton.
Hence we get $\{123\} \sim \{122\}$ since $p_n(\{123\})=p_n(\{122\})=2^{n-1}$ for all $n \geq 1$.

In this article, our focus will be on pairs of patterns $\{\sigma,\tau\}$, where $\sigma$ is a pattern of size $3$ with at least $2$ blocks.
Such pairs are called $(3,k)$-pairs when the the size of $\tau$ is $k$. 
Jel\'{\i}nek, Mansour, and Shattuck \cite{jel} studied such pairs and obtained an upper bound for the number of Wilf-equivalence classes of $(3,k)$-pairs. 
This was done by describing various Wilf-equivalences among $(3,k)$-pairs.
They also showed that any other Wilf-equivalences, if they exist, are between $(3,k)$-pairs of the form $\{112,\tau\}$ and left open the following problem.

\begin{problem*}[{\cite[Problem 2.17]{jel}}]
Are there any more equivalences among the $(3,k)$-pairs of the
form $\{112,\tau\}$ other than those that we know about? Equivalently, are there any two distinct 2-free integer partitions that are $\dsim$-equivalent (see \cref{defd})?
\end{problem*}

The main aim of this article is to answer this question.
In \cref{sec2}, we describe the results obtained by Jel\'{\i}nek, Mansour, and Shattuck \cite{jel} , and show how the second question is equivalent to the first in the above problem.
In \cref{sec3}, we answer this equivalent question and hence prove that there are no other Wilf-equivalences among $(3,k)$-pairs other than the ones described by Jel\'{\i}nek, Mansour, and Shattuck in \cite{jel}.

\section{Wilf-equivalences among $(3,k)$-pairs}\label{sec2}

In this section we will state the Wilf-equivalences between $(3,k)$-pairs derived by Jel\'{\i}nek, Mansour, and Shattuck \cite{jel}.

Note that if we are studying $P_n(T)$ where $T=\{\tau_1,\tau_2,\ldots\}$, we can assume that $\tau_i$ avoids $\tau_j$ for all $i \neq j$.
For example, when avoiding a $(3,k)$-pair $\{121,\tau\}$, we can assume that $\tau$ avoids $121$.
It can be shown that if the number of blocks and size of each block is specified, there is a unique partition that avoids $121$.
Namely, if a partition has $m$ blocks such that the $i^{th}$ block has $a_i$ elements for all $i \in [m]$, the unique such partition avoiding $121$ is $1^{a_1}2^{a_2}\cdots m^{a_m}$.
Here, $k^a$ represents $a$ consecutive copies of $k$.

\begin{definition}
A \textit{composition} of a positive integer $n$ is a sequence $(a_1,a_2,\ldots,a_m)$ such that $a_1+\cdots+a_m=n$.
The set of compositions of $n$ is denoted as $C_n$.
\end{definition}

For example, $C_3=\{(1,1,1),(1,2),(2,1),(3)\}$.
The above discussion implies that $P_n(\{121\})$ is in bijection with $C_n$.
The partition avoiding $121$ associated to the composition $a=(a_1,\ldots,a_m)$ is
\begin{equation*}
    \tau_{121}(a)=1^{a_1}2^{a_2}\cdots m^{a_m}.
\end{equation*}
It can be shown that the partitions of the form $\tau_{121}(b)$ that contain $\tau_{121}(a)$ are those where $b$ has a subsequence, having same length as $a$, such that each term in the subsequence has value at least that of the corresponding term in $a$.

\begin{definition}
The composition $b=(b_1,\ldots,b_k)$ is said to \textit{dominate} the composition $a=(a_1,\ldots,a_m)$ if there exists a subsequence $1\leq i(1) < i(2) < \cdots < i(m) \leq k$ such that $b_{i(j)} \geq a_j$ for all $j \in [m]$.
Such a subsequence is called an \textit{occurrence} of $a$ in $b$.
For any positive integer $n$, $D_n(a)$ will denote the set of compositions of $n$ that dominate $a$.
\end{definition}

Hence, we get that $P_n(\{121,\tau_{121}(a)\})=\{\tau_{121}(b) \mid b \in C_n \setminus D_n(a)\}$.
We now descibe the Wilf-equivalences among $(3,k)$-pairs of the form $\{121,\tau_{121}(a)\}$.
From the description of $P_n(\{121,\tau_{121}(a)\})$, it is clear that $\{121,\tau_{121}(a)\} \sim \{121,\tau_{121}(a')\}$ if and only if $|D_n(a)|=|D_n(a')|$ for all $n \geq 1$.
Hence we make the following definition.

\begin{definition}\label{defd}
For two compositions $a$ and $a'$, we say $a$ and $a'$ are \textit{dominating equivalent}, written as $a \dsim a'$, if for all positive integers $n$, $|D_n(a)|=|D_n(a')|$.
\end{definition}

Therefore, $\{121,\tau_{121}(a)\} \sim \{121,\tau_{121}(a')\}$ if and only if $a \dsim a'$.
Using similar ideas it can also be shown that $P_n(\{112\})$ is in bijection with compositions of $n$, where to a composition $a=(a_1,\ldots,a_m)$ we associate the partition
\begin{equation*}
    \tau_{112}(a)=12\cdots(m-1)mm^{a_m-1}\cdots2^{a_2-1}1^{a_1-1}.
\end{equation*}
Also, $P_n(\{112,\tau_{112}(a)\}) = \{\tau_{112}(b) \mid b \in C_n \setminus D_n(a)\}$ and $\{112,\tau_{112}(a)\} \sim \{112,\tau_{112}(a')\}$ if and only if $a \dsim a'$.

Using that fact that $p_k(\sigma)=2^{k-1}$ for all patterns $\sigma$ of size $3$ other than $111$ (see \cite[Theorem 4.3]{sag} for details), we get that if $k<k'$ then there are $2^{k-1}$ partitions of size $k$ avoiding a $(3,k')$-pair but only $2^{k-1}-1$ partitions of size $k$ avoiding a $(3,k)$-pair.
Hence, no $(3,k)$-pair can be Wilf-equivalent to a $(3,k')$-pair where $k \neq k'$.

A detailed discussion of the above facts can be found in the paper of Jel\'{\i}nek, Mansour, and Shattuck \cite{jel}.
In fact, they also showed that all except one Wilf-equivalence class of $(3,k)$-pairs correspond to dominating equivalence classes of $C_k$.

\begin{enumerate}
\setlength{\itemsep}{0.2cm}
    \item The Wilf-equivalence class corresponding to the dominating equivalence class $E$ containing $(1,\ldots,1)$ consists of the following pairs:
    \begin{enumerate}
    \setlength{\itemsep}{0.15cm}
        \item $\{121,\tau_{121}(a)\}$ where $a \in E$,
        \item $\{112,\tau_{112}(a)\}$ where $a \in E$,
        \item $\{122,\tau\}$ where $\tau \in P_k(\{122\})$, and
        \item $\{123,\tau'\}$ where $\tau' \neq 1^k$.
    \end{enumerate}
    \item The Wilf-equivalence class corresponding to a dominating equivalence class $E$ not containing $(1,\ldots,1)$ consists of the following pairs:
    \begin{enumerate}
    \setlength{\itemsep}{0.15cm}
        \item $\{121,\tau_{121}(a)\}$ where $a \in E$, and
        \item $\{112,\tau_{112}(a)\}$ where $a \in E$.
    \end{enumerate}
    \item The pair $\{123,1^k\}$ is not Wilf-equivalent to any other $(3,k)$-pair.
\end{enumerate}

Hence, if we denote the number of dominating equivalence classes in $C_k$ as $\xi_k$, we get that $(3,k)$-pairs split up into $1+\xi_k$ Wilf-equivalence classes.
Hence, we now shift our focus to the dominating equivalence among compositions.

\section{Dominating equivalence}\label{sec3}

We first recall a few results from the paper of
Jel\'{\i}nek, Mansour, and Shattuck \cite{jel}.

\begin{lemma}[{\cite[Lemma 2.3]{jel}}]\label{lemshuf}
If $a=(a_1,\ldots,a_m)$ and $a'=(a_{\sigma(1)},\ldots,a_{\sigma(m)})$ where $\sigma$ is a permutation of $[m]$, then $a \dsim a'$.
\end{lemma}

\begin{lemma}[{\cite[Lemma 2.4]{jel}}]\label{lemcomb11}
If $a=(a_1,\ldots,a_{m-1},2)$ and $a'=(a_1,\ldots,a_{m-1},1,1)$, then $a \dsim a'$.
\end{lemma}

These lemmas imply that to each composition we can associate a unique $2$-free integer partition which is dominating equivalent to it.
This is done by first rearranging the terms in decreasing order and then replacing each $2$ in the composition by two $1$s.

\begin{example}
The composition $(3,2,5,1,1,3)$ is converted to the associated unique $2$-free integer partition as follows:
\begin{equation*}
    (3,2,5,1,1,3) \dsim (5,3,3,2,1,1) \dsim (5,3,3,1,1,1,1).
\end{equation*}
\end{example}

Hence, any two compositions that correspond to the same $2$-free integer partition are dominating equivalent.
We will now show that the converse is true as well, which was conjectured by Jel\'{\i}nek, Mansour, and Shattuck \cite{jel}.

\begin{theorem}\label{thm2free}
Two compositions $a$ and $a'$ are dominating equivalent if and only if they correspond to the same $2$-free integer partition.
\end{theorem}

The above theorem implies that $\xi_k$, i.e., the number of dominating equivalence classes in $C_k$, is $p(k)-p(k-2)$, where for any $n \geq 1$, $p(n)$ is the number of integer partitions of $n$.
The sequence $(\xi_k)_{k\geq 0}$ is listed in the OEIS \cite{oeis} as \href{http://oeis.org/A027336}{A027336}.
Hence the number of Wilf-equivalences classes of $(3,k)$-pairs is $p(k)-p(k-2)+1$.

To prove the theorem, we first need a few lemmas and definitions.

\begin{lemma}\label{lem1}
If $a \dsim a'$ for two compositions $a$ and $a'$ of $m$ and $n$ respectively, then $m=n$.
\end{lemma}
\begin{proof}
Suppose to the contrary that $m$ and $n$ are distinct.
Without loss of generality, we assume $m<n$.
Then $D_m(a)=\{a\}$ whereas $D_m(a')$ is an empty set.
This contradicts the fact that $a \dsim a'$ and hence proves our lemma.
\end{proof}

\begin{definition}
The composition $b=(b_1,\ldots,b_k)$ is said to \textit{h-dominate} the composition $a$ if $b$ dominates $a$ and either $k=1$ or $(b_1,\ldots,b_{k-1})$ does not dominate $a$.
For any positive integer $n$, $D^h_n(a)$ is the set of compositions of $n$ that h-dominate $a$.
We say $a$ and $a'$ are \textit{h-dominating equivalent}, written as $a \hsim a'$, if for all positive integers $n$, $|D_n^h(a)|=|D_n^h(a')|$.
\end{definition}

We will now show that dominating equivalence is the same as h-dominating equivalence.
Before doing so, we set up some notations.
Let $c=(c_1,\ldots,c_k)$ and $c'=(c'_1,\ldots,c'_{k'})$ be compositions.
We will denote the composition $(c_1,\ldots,c_k,c'_1,\ldots,c'_{k'})$ by $(c,c')$.
Similarly, for any positive integer $s$, denote the composition $(c_1,\ldots,c_k,s)$ by $(c,s)$.

\begin{lemma}\label{lem2}
For any two compositions $a$ and $a'$, $a \dsim a'$ if and only if $a \hsim a'$.
\end{lemma}
\begin{proof}
Note that for any composition $a$,
\begin{equation}\label{eq1}
    D_n(a) = D_n^h(a)\ \sqcup\ \bigsqcup_{i=1}^{n-1} \{(c,c') \mid c \in D_i^h(a),\ c' \in C_{n-i}\}.
\end{equation}
This is obtained by finding, for each $b=(b_1,\ldots,b_k) \in D_n(a)$, the least $j \in [k]$ such that $(b_1,\ldots,b_j)$ dominates $a$.
This element $b$ of $D_n(a)$ belongs to exactly one of the $n$ sets on the right side.
If $j=k$, then $b \in D_n^h(a)$ or else $b \in \{(c,c') \mid c \in D_i^h(a),\ c' \in C_{n-i}\}$ where $i=b_1+\cdots+b_j$.

Suppose $a \hsim a'$.
The equality \eqref{eq1} implies that for all positive integers $n$,
\begin{align*}
    |D_n(a)| &= |D_n^h(a)|\ +\ \sum_{i=1}^{n-1} |D_i^h(a)|\cdot|C_{n-i}|\\
    &= |D_n^h(a')|\ +\ \sum_{i=1}^{n-1} |D_i^h(a')|\cdot|C_{n-i}|\\
    &= |D_n(a')|.
\end{align*}
Hence we get $a \dsim a'$.

Conversely, suppose $a \dsim a'$.
We will prove $|D_n^h(a)|=|D_n^h(a')|$ for all $n \geq 1$ by induction on $n$.
We have $|D_1^h(a)|=|D_1(a)|=|D_1(a')|=|D^h_1(a')|$.
Let $n \geq 2$ and $|D_k^h(a)| = |D_k^h(a')|$ for all $k<n$.
Using \eqref{eq1}, we get
\begin{align*}
    |D_n^h(a)| &= |D_n(a)|\ -\ \sum_{i=1}^{n-1} |D_i^h(a)|\cdot|C_{n-i}|\\
    &= |D_n(a')|\ -\ \sum_{i=1}^{n-1} |D_i^h(a')|\cdot|C_{n-i}|\\
    &= |D_n^h(a')|.
\end{align*}
Hence we get $a \hsim a'$, which completes the proof of the lemma.
\end{proof}

\begin{lemma}\label{lem3}
If $a$ and $a'$ are two compositions and $s$ is a positive integer such that $(a,s) \hsim (a',s)$, then we have $a \hsim a'$.
\end{lemma}
\begin{proof}
Let $a$ be a composition and $s$ be a positive integer.
Let $C_{k,s}$ denote those compositions of $k$ of the form $(c_1,\ldots,c_m)$ where $c_i<s$ for all $i \in [m-1]$ and $c_m \geq s$.
Note that
\begin{equation}\label{eq2}
    D^h_n((a,s)) = \bigsqcup_{i=1}^{n-1}\{(c,c') \mid c \in D^h_i(a),\ c' \in C_{n-i,s}\}.
\end{equation}
This is obtained by finding, for each $b=(b_1,\ldots,b_m) \in D^h_n((a,s))$, the least $j \in [m]$ such that $(b_1,\ldots,b_j)$ dominates $a$.
Since $b \in D^h_n((a,s))$, this would mean that $b_{j+1},\ldots,b_{m-1}$ are all strictly less than $s$ and $b_m \geq s$.
Note that $|C_{i,s}|=0$ for all $i < s$ and that $|C_{s,s}|=1$.
Hence an implication of the equality \eqref{eq2} is that for any positive integer $n$,
\begin{equation}\label{eq3}
    |D_{n+s}^h((a,s))| = |D_n^h(a)| + \sum_{i=s+1}^{n+s-1}|D_{n+s-i}^h(a)| \cdot |C_{i,s}|.
\end{equation}
Let $a'$ be a composition such that $(a,s) \hsim (a',s)$.
We will prove $|D_n^h(a)|=|D_n^h(a')|$ for all $n \geq 1$ by induction on $n$.
If $|D_1^h(a)| \neq |D_1^h(a')|$, we would have exactly one of $a$ or $a'$ being the composition of $1$.
This would mean $(a,s)$ and $(a',s)$ are compositions of different numbers, which contradicts \cref{lem1} (since \cref{lem2} implies that $(a,s) \dsim (a',s)$).
Let $n \geq 2$ and $|D_k^h(a)| = |D_k^h(a')|$ for all $k<n$.
From \eqref{eq3}, we get
\begin{align*}
    |D_n^h(a)| &= |D_{n+s}^h((a,s))| - \sum_{i=s+1}^{n+s-1}|D_{n+s-i}^h(a)| \cdot |C_{i,s}|\\
    &= |D_{n+s}^h((a',s))| - \sum_{i=s+1}^{n+s-1}|D_{n+s-i}^h(a')| \cdot |C_{i,s}|\\
    &= |D_n^h(a')|.
\end{align*}
Hence we get $a \hsim a'$, which completes the proof of the lemma.
\end{proof}

\begin{corollary}\label{cor}
If $a$ and $a'$ are two compositions and $s$ is a positive integer such that $(a,s) \dsim (a',s)$, then we have $a \dsim a'$.
\end{corollary}
\begin{proof}
Combine \cref{lem2} and \cref{lem3}.
\end{proof}

We can now prove our main theorem.

\begin{proof}[Proof of \cref{thm2free}]
We have to show that if $a$ and $b$ are two compositions such that $a \dsim b$, then they correspond to the same $2$-free integer partition.
By \cref{lem1}, such $a$ and $b$ are compositions of the same number $n$.
We will prove the theorem by induction on $n$, the case $n=1$ being trivial.
Let $n \geq 2$ and suppose the statement is true for all numbers less than $n$.
Suppose $a \dsim b$ are two compositions of $n$.
Using \cref{lemshuf} and \cref{lemcomb11}, we can assume that $a_1 \geq \cdots \geq a_m$ and $a_{m-1} \geq 2$ and similarly that $b_1 \geq \cdots \geq b_k$ and $b_{k-1} \geq 2$ (this can be done by reordering the terms in decreasing order and then changing any pair of $1$'s to a $2$).

We will now compute $|D_{n+1}(a)|$.
If $a_m \geq 2$, we get $|D_{n+1}(a)|=2m+1$.
This is because any composition in $D_{n+1}(a)$ is obtained by either
\begin{enumerate}
    \item inserting the term $1$ before $a_1$, between $a_i$ and $a_{i+1}$ for some $i \in [m-1]$, or after $a_m$, or
    \item adding $1$ to some term of $a$.
\end{enumerate}
Since all terms of $a$ are greater than $1$ in this case, each method of obtaining a composition of $D_{n+1}(a)$ described above results in a different composition.

When $a_m=1$, using the same logic as above but noting that adding $1$ before $a_m$ and adding $1$ after $a_m$ both result in the same composition, we get $|D_{n+1}(a)|=2m$.

Using similar arguments for $b$ and since we must have $|D_{n+1}(a)|=|D_{n+1}(b)|$, we get $k=m$, i.e., $a$ and $b$ have the same number of terms.

Note that if we show $a_m=b_m$, \cref{lem3} and the induction hypothesis would imply that $a$ and $b$ have the same corresponding $2$-free integer partition.
On the contrary, let $a_m \neq b_m$.
Without loss of generality we can assume that $a_m < b_m$.
We will show that this implies that $|D_{n+a_m}(a)| < |D_{n+a_m}(b)|$, which is a contradiction to $a \dsim b$.

A \textit{constructive pair} is a pair $(c,P)$ where $c=(c_1,\ldots,c_l)$ is a composition of $a_m$ and $P$ is a sequence that consists of the numbers $1,2,\ldots,l$ in order and $m$ boxes such that each box contains at most one of the numbers.
To each such pair $(c,P)$ we associate the composition in $D_{n+a_m}(b)$ obtained by replacing each unboxed number $i$ in the sequence $P$ by $c_i$ and the $j^{th}$ box by $b_j$ if it is empty and by $b_j+c_i$ if it contains the number $i$.
We will call this composition $b(c,P)$.

\begin{example}\label{egconst}
Suppose $b=(8,6,6,5)$, $a_m=4$, $c=(1,2,1)$, and $P=1\ \framebox{\color{white}2}\ \framebox{2}\ 3\ \framebox{\color{white}2}\ \framebox{\color{white}2}$.
Then the composition $b(c,P)$ is $(1,8,8,1,6,5)$.

The construction of $b(c,P)$ from $b$ be thought of visually as follows:
Think of any composition as towers of boxes.
The boxes in the sequence $P$ can be thought of as the top view of $b$ and the number $i$ represents where to add $c_i$ boxes to $b$.
The construction in the example is shown in \Cref{fig:my_label}.
\end{example}

\begin{figure}[H]
    \centering
    \begin{tikzpicture}[scale=0.5]
        \draw (0,0) rectangle (0+1,0+1);
        \draw (0,1) rectangle (0+1,1+1);
        \draw (0,2) rectangle (0+1,2+1);
        \draw (0,3) rectangle (0+1,3+1);
        \draw (0,4) rectangle (0+1,4+1);
        \draw (0,5) rectangle (0+1,5+1);
        \draw (0,6) rectangle (0+1,6+1);
        \draw (0,7) rectangle (0+1,7+1);
        
        \draw (1,0) rectangle (1+1,0+1);
        \draw (1,1) rectangle (1+1,1+1);
        \draw (1,2) rectangle (1+1,2+1);
        \draw (1,3) rectangle (1+1,3+1);
        \draw (1,4) rectangle (1+1,4+1);
        \draw (1,5) rectangle (1+1,5+1);
        
        \draw (2,0) rectangle (2+1,0+1);
        \draw (2,1) rectangle (2+1,1+1);
        \draw (2,2) rectangle (2+1,2+1);
        \draw (2,3) rectangle (2+1,3+1);
        \draw (2,4) rectangle (2+1,4+1);
        \draw (2,5) rectangle (2+1,5+1);
        
        \draw (3,0) rectangle (3+1,0+1);
        \draw (3,1) rectangle (3+1,1+1);
        \draw (3,2) rectangle (3+1,2+1);
        \draw (3,3) rectangle (3+1,3+1);
        \draw (3,4) rectangle (3+1,4+1);
        \node at (6,3) {$\xrightarrow{(c,P)}$};
        \node at (2,-1) {$\textcolor{white}{(P}b\textcolor{white}{P)}$};
    \end{tikzpicture}
    \begin{tikzpicture}[scale=0.5]
        \draw[blue] (-1,0) rectangle (-1+1,0+1);
        \node at (-0.5,0.5) {\color{blue}$c_1$};
        \draw (0,0) rectangle (0+1,0+1);
        \draw (0,1) rectangle (0+1,1+1);
        \draw (0,2) rectangle (0+1,2+1);
        \draw (0,3) rectangle (0+1,3+1);
        \draw (0,4) rectangle (0+1,4+1);
        \draw (0,5) rectangle (0+1,5+1);
        \draw (0,6) rectangle (0+1,6+1);
        \draw (0,7) rectangle (0+1,7+1);
        
        \draw (1,0) rectangle (1+1,0+1);
        \draw (1,1) rectangle (1+1,1+1);
        \draw (1,2) rectangle (1+1,2+1);
        \draw (1,3) rectangle (1+1,3+1);
        \draw (1,4) rectangle (1+1,4+1);
        \draw (1,5) rectangle (1+1,5+1);
        \draw[blue] (1,6) rectangle (1+1,6+2);
        \node at (1.5,7) {\color{blue}$c_2$};
        
        \draw[blue] (2,0) rectangle (2+1,0+1);
        \node at (2.5,0.5) {\color{blue}$c_3$};
        \draw (2+1,0) rectangle (2+1+1,0+1);
        \draw (2+1,1) rectangle (2+1+1,1+1);
        \draw (2+1,2) rectangle (2+1+1,2+1);
        \draw (2+1,3) rectangle (2+1+1,3+1);
        \draw (2+1,4) rectangle (2+1+1,4+1);
        \draw (2+1,5) rectangle (2+1+1,5+1);
        
        \draw (3+1,0) rectangle (3+1+1,0+1);
        \draw (3+1,1) rectangle (3+1+1,1+1);
        \draw (3+1,2) rectangle (3+1+1,2+1);
        \draw (3+1,3) rectangle (3+1+1,3+1);
        \draw (3+1,4) rectangle (3+1+1,4+1);
        \node at (3,-1) {$b(c,P)$};
        \end{tikzpicture}
    \caption{Construction corresponding to the constructive pair given in \cref{egconst}}
    \label{fig:my_label}
\end{figure}
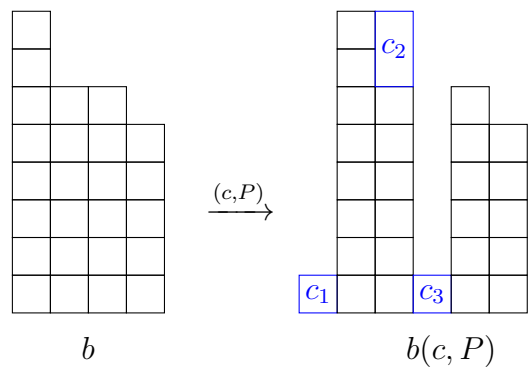



Using any occurrence of $b$ in a composition $b'$ in $D_{n+a_m}(b)$, we can obtain a constructive pair $(c,P)$ such that $b(c,P)=b'$.
This is done as follows:
Suppose $b'=(b'_1,\ldots,b'_p)$ is a composition in $D_{n+a_m}(b)$ and $1 \leq i(1) < i(2) < \cdots < i(m) \leq p$ is an occurrence of $b$.
Construct $P$ from $b'$ by first replacing $b'_j$ by a ball if $j \notin \{i(1),\ldots,i(m)\}$.
Then for any $j \in [m]$, replace $b'_{i(j)}$ by an empty box if $b'_{i(j)}=b_j$ and by a box containing a ball if $b'_{i(j)}>b_j$.
Suppose there are $l$ balls in $P$, we replace them by the numbers $1,\ldots,l$ in order.
Now for each $i \in [l]$,
\begin{enumerate}
    \item set $c_i$ as $b'_j$ if $i$ is an unboxed number and the $j^{th}$ term in $P$, and
    \item set $c_i$ as $b'_{i(j)}-b_j$ if $i$ is a boxed number and the $i(j)^{th}$ term of $P$.
\end{enumerate}
The above procedure might be more clear when compositions are thought of as towers of boxes.
For example, in \Cref{fig:my_label}, if the black boxes in $b(c,P)$ are used as the occurrence of $b$ in the above procedure, we obtain the constructive pair $(c,P)$ of \cref{egconst}.

It is clear that if $b'=b(c,P)$ for some constructive pair $(c,P)$, then there is a canonical occurrence of $b$ in $b'$ formed using the original copy of $b$ to which terms were added in the construction.
Using this canonical occurrence in the above procedure we can get back the original constructive pair $(c,P)$.
However, since $b$ has all terms strictly greater than $a_m$, there is a unique occurrence of $b$ in any composition $b'$ in $D_{n+a_m}(b)$.
Hence, there is a unique constructive pair $(c,P)$ such that $b(c,P)=b'$ for any composition $b'$ of $D_{n+a_m}(b)$.
Therefore, the set $D_{n+a_m}(b)$ is in bijection with the set of constructive pairs.

Since $a$ also has $m$ terms, the same method can be used to associate a composition in $D_{n+a_m}(a)$ to a constructive pair $(c,P)$, and we call this composition $a(c,P)$.
Just as for $b$, any composition $a' \in D_{n+a_m}(a)$ can be obtained from $a$ using a constructive pair.
However, $a((a_m),P_1)=a((a_m),P_2)$ where $P_1$ has $1$ between the $(m-1)^{th}$ and $m^{th}$ box and $P_2$ has $1$ after the $m^{th}$ box.
Hence we get $|D_{n+a_m}(a)| < |D_{n+a_m}(b)|$, which completes the proof.
\end{proof}

\bibliographystyle{abbrv}
\bibliography{refs}

\end{document}